\documentclass[]{amsart}
\usepackage{amscd,amsthm,amssymb,amsfonts,amsmath,euscript}

\theoremstyle{plain}
\newtheorem{thm}{Theorem}[section]
\newtheorem{lemma}[thm]{Lemma}
\newtheorem{prop}[thm]{Proposition}
\newtheorem{cor}[thm]{Corollary}

\theoremstyle{definition}
\newtheorem{defn}[thm]{Definition}

\theoremstyle{remark}
\newtheorem{remark}[thm]{Remark}

\newcommand{\nc}{\newcommand}

\def\makeop#1{\expandafter\def\csname#1\endcsname
  {\mathop{\rm #1}\nolimits}\ignorespaces}
\makeop{Hom}   \makeop{End}   \makeop{Aut}   \makeop{Isom}  \makeop{Pic} 
\makeop{Gal}   \makeop{ord}   \makeop{Char}  \makeop{Div}   \makeop{Lie} 
\makeop{PGL}   \makeop{Corr}  \makeop{PSL}   \makeop{sgn}   \makeop{Spf}
\makeop{Spec}  \makeop{Tr}    \makeop{Nr}    \makeop{Fr}    \makeop{disc}
\makeop{Proj}  \makeop{supp}  \makeop{ker}   \makeop{im}    \makeop{dom}
\makeop{coker} \makeop{Stab}  \makeop{SO}    \makeop{SL}    \makeop{SL}
\makeop{Cl}    \makeop{cond}  \makeop{Br}    \makeop{inv}   \makeop{rank}
\makeop{id}    \makeop{Fil}   \makeop{Frac}  \makeop{GL}    \makeop{SU}
\makeop{Nrd}   \makeop{Sp}    \makeop{Tr}    \makeop{Trd}   \makeop{diag}
\makeop{Res}   \makeop{ind}   \makeop{depth} \makeop{Tr}    \makeop{st}
\makeop{Ad}    \makeop{Int}   \makeop{tr}    \makeop{Sym}   \makeop{can}
\makeop{length}\makeop{SO}    \makeop{torsion} \makeop{GSp} \makeop{Ker}
\makeop{Adm}   \makeop{Mat}
\def\makebb#1{\expandafter\def
  \csname bb#1\endcsname{{\mathbb{#1}}}\ignorespaces}
\def\makebf#1{\expandafter\def\csname bf#1\endcsname{{\bf
      #1}}\ignorespaces} 
\def\makegr#1{\expandafter\def
  \csname gr#1\endcsname{{\mathfrak{#1}}}\ignorespaces}
\def\makescr#1{\expandafter\def
  \csname scr#1\endcsname{{\EuScript{#1}}}\ignorespaces}
\def\makecal#1{\expandafter\def\csname cal#1\endcsname{{\mathcal
      #1}}\ignorespaces} 

\def\doLetters#1{#1A #1B #1C #1D #1E #1F #1G #1H #1I #1J #1K #1L #1M
                 #1N #1O #1P #1Q #1R #1S #1T #1U #1V #1W #1X #1Y #1Z}
\def\doletters#1{#1a #1b #1c #1d #1e #1f #1g #1h #1i #1j #1k #1l #1m
                 #1n #1o #1p #1q #1r #1s #1t #1u #1v #1w #1x #1y #1z}
\doLetters\makebb   \doLetters\makecal  \doLetters\makebf
\doLetters\makescr 
\doletters\makebf   \doLetters\makegr   \doletters\makegr

     \def\qed{\qedmark\medbreak}%
\def\qedmark{{\enspace\vrule height 6pt width 5pt depth 1.5pt}}%
    
\def\Gm{{{\bbG}_{\rm m}}}   

\normalsize

\makeop{Bl}

\def\Spec{{\rm Spec}\,}

\def\Fp{{\bbF}_p}

\def\char{{\rm char\,}}

\newcommand{\Z}{\mathbb Z}
\newcommand{\Q}{\mathbb Q}

\newcommand{\F}{\mathbb F}

\newcommand{\<}{\langle}   
\renewcommand{\>}{\rangle} 

\newcommand{\isoto}{\stackrel{\sim}{\longrightarrow}}
\nc{\embed}{\hookrightarrow}

\newcommand{\ch}{characteristic }

\nc{\ol}{\overline}
\nc{\wt}{\widetilde}
\nc{\opp}{\mathrm{opp}}

\makeop{Ram}
\makeop{Rep}

\begin{document}
\renewcommand{\thefootnote}{\fnsymbol{footnote}}
\setcounter{footnote}{-1}
\numberwithin{equation}{section}


\title[Chow's theorem and algebraic tori]
{Chow's theorem for semi-abelian varieties and 
bounds for splitting fields of algebraic tori}
\author{Chia-Fu Yu}
\address{
Institute of Mathematics, Academia Sinica\\
Astronomy Mathematics Building \\
No.~1, Roosevelt Rd. Sec.~4 \\ 
Taipei, Taiwan, 10617} 
\email{chiafu@math.sinica.edu.tw}

\address{
 National Center for Theoretical Sciences
 No.~1  Roosevelt Rd. Sec.~4,
 National Taiwan University
 Taipei, Taiwan, 10617}


\date{\today}


\begin{abstract}
A theorem of Chow concerns homomorphisms of two abelian varieties under a
primary field extension base change. In this paper we generalize
Chow's theorem to semi-abelian varieties. This contributes to different
proofs of a well-known result that every algebraic torus splits over a
finite separable field extension. We also obtain 
the best bound for the degrees of splitting fields of tori.

\end{abstract} 

\maketitle


\section{Introduction}
\label{sec:introduction}
Let $k$ be a field, $\bar k$ an algebraic closure of $k$, and $k_s$
the separable closure of $k$ in $\bar k$. 
A connected algebraic $k$-group $T$ is an algebraic torus
if there is a $\bar k$-isomorphism 
$T\otimes_k \bar k \simeq (\Gm)^d \otimes
  _k \ol k$ of algebraic groups for some integer $d\ge 0$.
We say $T$ splits over a 
field extension $K$ of $k$ if there is a $K$-isomorphism
$T\otimes_k K\simeq (\Gm)^d\otimes_k K$. 
This paper is motivated from the following fundamental result.

\begin{thm}\label{1.1}
  Any algebraic $k$-torus $T$ splits over $k_s$. In other
  words, $T$ splits over a finite separable field extension of $k$.
\end{thm}

This theorem is well known and it is stated and proved in the 
literature several times. Surprisingly, different authors 
choose their favorite proofs which are all quite different. 
As far as we know, the first proof is given
by Takashi Ono \cite[Proposition 1.2.1]{ono:ann1961}. 
Armand Borel gives a different proof in his book 
{\it Linear Algebraic Groups}; see
\cite[Proposition 8.11]{borel:lag}. 
In  the second edition of 
his book {\it Linear Algebraic Groups} \cite{springer:lag}, 
T.A.~Springer includes a systematic treatment of the rationality problem 
of algebraic groups where he also gives another proof 
of Theorem~\ref{1.1}; see \cite[Proposition 13.1.1]{springer:lag}. 
Another proof, due to John Tate, can be found in Borel
and Tits \cite[Proposition 1.5]{borel-tits}. 
Jacques Tits himself also provides one proof 
in his Yale University Lectures Notes; 
see \cite[Theorem 1.4.1]{tits:yale1970}. 

A key point in Borel's proof 
is 
that 
any $k_s$-valued point of $T$ is semi-simple. 
In fact, Borel proves Theorem~\ref{1.1} more generally 
for diagonalizable $k$-groups, because 
this property also holds for diagonalizable groups.
%
Springer's proof has more flavor of differential geometry; a key
ingredient 
uses derivations. 
In some sense Springer treats a purely inseparable descent
using derivations and connections though this is not stated
explicitly.
It is  known (cf.~Matsumura~\cite[Chap.~9]{matsumura:crt86}) 
that some hidden information of purely
inseparable field extensions ignored by Galois theory 
can be revealed by derivations and connections. However, 
the way that Springer applies such an inseparable descent to
Theorem~\ref{1.1} is rather interesting.    
Tits' and Tate's proofs use only algebraic properties of characters;
the ideas of their proofs are similar. 
Both start with that a suitable
$p$-power of a character $\chi$ of $T$ is defined over $k_s$ and
prove that $\chi$ is defined over $k_s$. 
A main difference is that Tits works with the coordinate
ring $\bar k [T]$ of $T$ while Tate works with its function field
$\bar k(T)$. That Tate proves the $k_s$-rationality of $\chi$ 
uses the language in Weil's foundation, while 
Tits' argument is more elementary. We shall present only 
the proofs of Ono and Borel as they are most relevant to our results. 

Chow's theorem \cite{chow:tams1955} states as follows: 
Let $K/k$ be a primary field extension,
that is, $k$ is separably algebraically closed in $K$. If $X$ and $Y$
are two abelian varieties over $k$, then the map
$\Hom_k(X,Y)\to \Hom_K(X_K, Y_K)$ is bijective, 
where $X_K:=X\otimes_k K$ and $Y_K:=Y\otimes_k K$. 
T.~Ono observes \cite[Lemma 1.2.1]{ono:ann1961} that 
Theorem~\ref{1.1} follows immediately from 
an analogue of Chow's theorem for tori (i.e. the above bijection also
holds if one replaces abelian varieties by tori), 
and he points out in the proof
that Chow's original proof (for abelian varieties) also works for tori.
As is well known, Chow's theorem (also see \cite[Chapter
II, Theorem 5]{lang:av}) is proved under an old fashion of Weil's
foundation.
Thus, it is desirable to have a 
modern proof using the
language of schemes for Chow's theorem and its analogue for tori.
Brian Conrad \cite[Theorem 3.19]{conrad:Kktrace} gives 
a modern proof of the original Chow theorem.
The central idea is Grothendieck's faithfully flat
descent.

Grothendieck's descent theory has been a very powerful tool of
algebraic geometry.
The standard reference is SGA 1~\cite{SGA1}. The reader can also find the
exposition in some books or articles working with moduli spaces or
\'etale cohomology, for example, 
Milne~\cite[Chapter 1, Section 2]{milne:ec}, 
Freitag and Kiehl \cite[Appendix A]{freitag-kiehl}, 
Bosch, L\"utkebohmert, and 
Raynaud~\cite[Chapter 6]{bosch-luetkebohmert-raynaud} and 
B.~Conrad~\cite[Section 3]{conrad:Kktrace}.   
The faithfully flat descent 
is a very clean formulation which reorganizes both
the classical Galois descent and the purely inseparable descent
through derivations over fields in one unified way (regardless the
explicit structure of the flat base in question).  
More powerfully, this simple formulation works for arbitrary base
schemes, so it is far beyond the combination of both 
separable and inseparable descent over fields. 

The idea of Conrad's proof of Chow's theorem is pursued further 
in this article. 
Indeed, using the similar idea, we generalize Chow's theorem to 
semi-abelian varieties.
This includes the case of tori which is related to  
Theorem~\ref{1.1} by our discussion.
We refer to Section~\ref{sec:chow.2} for the definition of 
semi-abelian varieties. 

\begin{thm}\label{1.2}
  Let $X$ and $Y$ be two semi-abelian varieties over a field $k$, and
  let $K$ be a primary field extension of $k$. Then the monomorphism
  of $\Z$-modules 
  \begin{equation}
    \label{eq:1.1}
    \Hom_k(X,Y)\to \Hom_K (X_K, Y_K) 
  \end{equation}
is bijective.
\end{thm}
       
We actually give two proofs of this theorem. As mentioned, the first
proof utilizes the flat descent. Our second proof is more elementary; it
does not rely on the flat descent. Thus, we add another two (or one
depending on how one counts) proofs of Theorem~\ref{1.1} to 
those given by Borel, Springer, 
Tate, Ono and Tits. We remark that Theorem~\ref{1.1} is equivalent to
Theorem~\ref{1.2} for tori; see Proposition~\ref{proofs.1}.




We illustrate how Theorem~\ref{1.1} is related to integral
representations of finite groups, the inverse Galois problem and 
Noether's problem. These are 
interesting research topics and remain very active even up to date.
Let $\Gamma_k:=\Gal(k_s/k)$ be the absolute Galois group of $k$. Using
Theorem~\ref{1.1}, the functor 
\begin{equation}
  \label{eq:1}
  \text{a $k$-torus $T$} \ \mapsto X(T):=\Hom_{k_s}( 
  T_{k_s},{\Gm}_{k_s})
\end{equation}
gives rise to an anti-equivalence of categories
between that of algebraic $k$-tori and that of
finitely generated free $\Z$-modules together with a continuous
action of $\Gamma_k$, or simply $\Z\Gamma_k$-lattices. 
The inverse functor is given by 
$M\mapsto \Spec k_s[M]^{\Gamma_k}$.
If $K/k$ is a finite Galois extension with
Galois group $G:=\Gal(K/k)$, then this functor induces a bijection
\begin{equation}
  \label{eq:1.2}
  \left\{ \ \parbox{1.4in}{isomorphism classes of $k$-tori splitting
  over $K$} \ \right \} \isoto \left\{\,
  \parbox{1.2in}{isomorphism classes of $\Z
      G$-lattices}\, \right \}. 
\end{equation}
Thus, classifying $k$-tori can be divided into two steps:
\begin{itemize}
\item [(1)] Classify all finite groups $G$ which appear as quotients of
  $\Gamma_k$.
\item [(2)] Classify integral representations of $G$.
\end{itemize}

It is known that the symmetric group $S_n$ or the alternating group
$A_n$ occurs as a Galois group over $\Q$ \cite[Chapter 4]{serre:Gal}.
A theorem of Shafarevich states that if $k$ is a global field and $G$
is a finite solvable group, then there exists a finite Galois
extension $K/k$ with $\Gal(K/k)$ isomorphic to $G$; see 
\cite[Theorem 9.6.1]{neukirsch-schmidt-wingberg:2ed} (or Theorem 9.5.1 in
the 1st edition). Zywina \cite{zywina:psl2} 
proves that the finite simple group 
$\rm PSL_2(\Fp)$ for $p\ge 5$ occurs as a Galois group over $\Q$. 
We remind the reader the book by Jensen, Ledet and Yui 
\cite{jensen-ladet-yui}
where one can find 
many explicit examples of generic polynomials for small or medium
groups. This  
book also provides a convenient and detailed introduction 
to Galois theory and a nice exposition of Saltman's paper
\cite{saltman:1982}. We should also mention some published books
related to the subject, for example, those of Serre \cite{serre:Gal}, 
V\"olklein, 
and Malle and Matzat. 

Noether's problem is the most prominent 
and famous problem in the inverse Galois problem. 
For a finite group $G$, consider the field extension 
$k(G):=k(x_g|g\in G)^G$ of a field
$k$, where the action of $G$ on the finite set $\{x_g|g\in G\}$ of
indeterminates  is defined by $g\cdot x_h=x_{gh}$ for all
$g,h\in G$. Noether's problem asks whether $k(G)$ is rational (=purely
transcendental) over $k$.  
A standard survey of Noether's problem 
is Swan's paper~\cite{swan:noether}. 
For a more recent survey, see Kersten \cite{kersten}.
Lenstra \cite{lenstra:noether} gives a complete solution for the case
of finite abelian groups. 
D.~Saltman \cite{saltman:brauer_1984} introduces the 
unramified Brauer group $\Br_{\rm ur}(k(G))$ of $k(G)$, 
which paves a way  to construct
counterexamples to Noether's problem. If $k(G)$ is rational over $k$,
then $\Br_{\rm ur}(k(G))$ is trivial. Saltman
produces an example of a group $G$ of order $p^9$ where $p$ is a prime
number different from ${\rm char\,} k$ such that 
$\Br_{\rm ur}(k(G))$ is
non-trivial. F.~A.~Bogomolov \cite{bogomolov:brauer} gives 
an explicit formula for the unramified Brauer group and produces a similar
example, with $G$ of order $p^6$. 
The notion of Saltman's unramified Brauer group 
is extended to the higher degree unramified cohomology
groups by Colliot-Th\'el\`ene and Ojanguren~\cite{CT-Ojanguren:1989}. 
For recent development of Noether's
problem, we refer to works of A.~Hoshi, 
M.-C. Kang, H.~Kitayama, and
A.~Yamasaki~ and others 
\cite{chu-hoshi-hu-kang, hoshi:noether2015, kang:adv2006,kang:imrn2009,
  kang:adv2011,kang:cyclic, Ki-Y, Ki:2010, plans, yamasaki:2010}. 
Surely, the reader can also easily find more in 
the literature. We also refer to the
excellent survey paper by B.~Kunyavskii~\cite{kunyavskii:30y} for
developments of more research topics related to algebraic tori.

We conclude this paper with the best bound for the degrees of 
splitting fields of tori.
\def\Max{{\rm Max}}
\begin{prop}(Corollary~\ref{spl.4})
For any $d\ge 1$ and any number field $k$, there exist infinitely many
$d$-dimensional tori $T$ over $k$ such that $[k_T:k]=\Max(d,\Q)$, where
$k_T$ is the (minimal) splitting field of $T$ and $\Max(d,\Q)$ is the
maximal order of finite subgroups of $\GL_d(\Q)$. 
\end{prop}

The paper is organized as follows. Section~\ref{sec:chact-diag-groups} 
contains minimal
preliminaries of diagonalizable groups and proofs of 
Theorem~\ref{1.1} due to A.~Borel and T.~Ono.
Theorem~\ref{1.2} is proved in Section~\ref{sec:chow}; 
we also give another proof of 
Theorem~\ref{1.2} and hence that of Theorem~\ref{1.1}. 
In Section~\ref{sec:spl} we
study bounds of splitting fields of algebraic tori.
 

\section{Two proofs of Theorem~\ref{1.1}}
\label{sec:chact-diag-groups}

\subsection{Characters and diagonalizable groups}\label{sec:ch.1}
In this subsection we shall present a proof of Theorem~\ref{1.1} due to
Armand Borel \cite{borel:lag}. 
A basic result is that any set of characters is linearly independent
in the following sense.
 
\begin{lemma}\label{ch.1}
  Let $H$ be an abstract group, $k$ any field, and let $X$ be the set of
  all homomorphisms $H\to k^\times$. Then $X$ is $k$-linearly
  independent as a subset in the $k$-vector space $\calC(H,k)$ of
  $k$-valued 
  functions on $H$. That is, for any distinct characters
  $\chi_1,\dots, \chi_n$ and elements $a_1, \dots, a_n\in k$, then 
  \begin{equation}
    \label{eq:ch.1}
    a_1 \chi_1+\dots + a_n \chi_n=0 \ \text{in $\calC(H,k)$} \implies
    a_1=\dots=a_n=0.
  \end{equation}
\end{lemma}
\begin{proof}
See \cite[Lemma 8.1]{borel:lag}.  
\end{proof}\ 

Let $G$ be a linear algebraic group over a field $k$. 
Put $K:=\bar k$. 
Let $X(G):=\Hom_{K\text{-gp}}(G,\Gm)$ denote the group of all
characters, which is a finitely generated abelian group. The 
subgroup of $k$-rational characters is denoted by $X(G)_k$.

\begin{defn}\label{ch.2}
  (1) We say that $G$ is {\it diagonalizable} if the coordinate ring 
  $K[G]$ is spanned by $X(G)$ over $K$. 

  (2) We say that a diagonalizable group $G$ splits over $k$ if the
  coordinate ring $k[G]$ is spanned by $X(G)_k$ over $k$.      
\end{defn}

If $H$ is an abstract group, 
then the group algebra $K[H]$ of $H$ over $K$ admits a 
natural structure of Hopf
algebra, with the co-multiplication $\Delta:K[H]\to K[H]\otimes_K
K[H]$ defined by $\Delta(h)=h\otimes h$.
By Lemma~\ref{ch.1}, the abelian group $X(G)$ is a linearly
independent subset in
$K[G]$. So if $G$ is diagonalizable, then $K[G]$ is equal to 
the group algebra $K[X(G)]$ of the abelian group $X(G)$. Moreover,
this equality respects the Hopf algebra structures of $K[G]$ and of
$K[X(G)]$. 
Particularly, $G$ is commutative if $G$ is diagonalizable.

It is clear from the definition that an algebraic torus is precisely a 
connected diagonalizable algebraic group. We recall basic properties
of diagonalizable groups. Let ${\bf D}_n\subset \GL_n$, for $n\ge 1$,
denote the diagonal split torus of dimension $n$.

\begin{prop}\label{ch.3}
  Let $G$ be a linear algebraic group over $K$. The following
  statements are equivalent:

  \begin{enumerate}
  \item $G$ is diagonalizable.
  \item $G$ is isomorphic to a subgroup of ${\bf D}_n$ for some $n\ge 1$.
  \item For each rational representation $\pi: G\to \GL_n$, the
    subgroup $\pi(G)$ is conjugate to a subgroup of ${\bf D}_n$.

  \item $G$ contains a dense commutative subgroup consisting of
    semi-simple elements. 
  \end{enumerate}
\end{prop}
\begin{proof}
  See \cite[Proposition 8.4]{borel:lag}.
\end{proof}


\begin{thm}\label{ch.5}
  Every diagonalizable $k$-group $G$ splits over $k_s$. 
\end{thm}
\begin{proof}
  Choose a $k$-embedding $G\subset \GL_n$. By Proposition~\ref{ch.3},
  $G(k_s)$ contains a dense commutative subgroup $S$ consisting of
  semi-simple elements. As $S$ is commutative and every element $s$ in
  $S$ is diagonalizable in $\GL_n(k_s)$, we can diagonalize
  simultaneously the matrices $s$ for all $s\in S$. 
  That is, there is an element $g\in
  \GL_n(k_s)$ such that $g S g^{-1}\subset {\bf D}_n(k_s)$. 
  Since $S$ is dense, the inner automorphism $\Int (g)$ sends $G$ 
  into a subgroup of ${\bf D_n}$. Therefore, $G$ is $k_s$-isomorphic
  to a subgroup of ${\bf D_n}$. Note that every subgroup of ${\bf D_n}$
  splits over $k$ and hence $G\otimes_k k_s$ is a split torus. \qed
\end{proof}

Theorem~\ref{1.1} follows from Theorem~\ref{ch.5}, because every
algebraic torus is a diagonalizable group. 
We remark that Proposition~\ref{ch.3} (4) is 
also known due to Rosenlicht who uses it to show  
that every algebraic $k$-torus is unirational; see \cite[Proposition
10]{rosenlicht:1957}.


\subsection{Ono's proof}\label{sec:ch.2}
We now present Ono's proof of Theorem~\ref{1.1} based on
Theorem~\ref{1.2} for tori. 
Let $T$ be an algebraic torus over $k$.
It suffices to show that any character $\chi\in X(T)_{\bar
  k}=\Hom_{\bar k}(T,\Gm)$ is defined over $k_s$.

Since $\bar k/k_s$ is primary, applying Theorem~\ref{1.2} to $A=T$ and
$B=\Gm$, we have an isomorphism
\[ \Hom_{k_s}(T,\Gm)\isoto \Hom_{\bar k}(T,\Gm). \]
Thus, every character is defined over $k_s$.  \qed



We observe that Theorem~\ref{1.1} and Theorem~{1.2} for tori are
equivalent. 

\begin{prop}\label{proofs.1}
  The following two statements are equivalent.

  \begin{enumerate}
  \item [(a)] Every $k$-torus splits over $k_s$.
  \item [(b)] For any two $k$-tori $X$ and $Y$ and any primary field
    extension $K/k$, we have the bijection $\Hom_k(X,Y)\simeq
    \Hom_K(X_K,Y_K)$. 
  \end{enumerate}
\end{prop}
\begin{proof}
  We have proven (b)$\implies$ (a) and now prove the other direction.
  Choose a separable closure $K_s$ of $K$ containing $k_s$. By (a), we
  have $\Hom_{k_s}(X_{k_s},Y_{k_s})\simeq
  \Hom_{K_s}(X_{K_s},Y_{K_s})=\Hom_{k_sK}(X_{k_sK},Y_{k_s K})$. As
  $K/k$ is primary, $K\cap k_s=k$ and $\Gal(k_sK/K)\simeq
  \Gamma_k=\Gal(k_s/k)$. Thus,
  \begin{equation}
    \label{eq:proofs.1}
    \begin{split}
      \Hom_K(X_K,Y_K)&=\Hom_{K_s}(X_{K_s},Y_{K_s})^{\Gamma_K}\\
      &=\Hom_{K
      k_s}(X_{Kk_s},Y_{Kk_s})^{\Gal(Kk_s/K)} \\
      &\simeq
      \Hom_{k_s}(X_{k_s},Y_{k_s})^{\Gamma_k}=\Hom_k(X,Y). 
      \quad \text{\qed} 
    \end{split}
  \end{equation}
\end{proof}

\section{Chow's theorem for semi-abelian varieties}
\label{sec:chow}

In this section we shall give a proof of Theorem~\ref{1.2}.
As mentioned in Section~\ref{sec:introduction}, 
the main ingredient is Grothendieck's
faithfully flat descent.

\subsection{Faithfully flat descent}
\label{sec:chow.1}

We recall some basic terminology needed to describe the flat descent.
\begin{defn}\label{chow.1}\ 

{\rm (1)} A ring homomorphism $A\to B$ of commutative rings 
  is said to be {\it flat} 
  if the functor $\otimes_A B: A\text{-Mod}\to
  B\text{-Mod}$ is exact.


{\rm (2)} A morphism $f:X\to Y$ of schemes 
   is said to be {\it flat} if for any point $x\in X$ with $y=f(x)$,
   the ring homomorphism $\calO_{Y,y}\to \calO_{X,x}$ of local rings
   is flat. We say $f$ is {\it faithfully flat} if it is flat and
   surjective. 

{\rm (3)} We say $f$ is {\it quasi-compact} if the pre-image 
  $f^{-1}(U)$ of every open affine subscheme $U$ of $Y$ is
  quasi-compact, that is, it is a finite union of open affine
  subschemes. 

{\rm (4)} A scheme $X$ is said to be {\it quasi-affine} if it is
quasi-compact and it is contained in an affine scheme. A morphism $f$
of schemes is said to be {\it quasi-affine} if it is quasi-compact and
the  pre-image 
  $f^{-1}(U)$ of every open affine subscheme $U$ of $Y$ is
  quasi-affine.

{\rm (5)} Let $Y$ be a Noetherian scheme. A morphism $f:X\to
Y$ of schemes of finite type is said to be {\it projective}
(resp.~ quasi-projective) if $X$ is isomorphic to a closed
(resp.~ locally closed) subscheme of the projective scheme $\bfP^N_Y$
for some positive integer $N$.     
\end{defn}

We first describe the flat descent for morphisms.
Let $p:S'\to S$ be a morphism of base schemes,
and let $X\to S$
be a morphism of schemes. For any integer $n>1$, 
write  $S^{(n)}:=S'\times_S
\dots \times_S S'$ ($n$ times), and $X^{(n)}:=X\times_S S^{(n)}$. 
Let $p_1,p_2: S'':=S^{(2)} \to S'$ be two projection maps. 

\begin{prop}\label{chow.2}
Let $p:S'\to S$ be a faithfully flat and quasi-compact 
morphism of base schemes.
Let $X$ and $Y$ be
two schemes over $S$ and let $f':X'\to Y'$ a morphism of schemes over 
$S'$. If $p_1^*(f')=p_2^*(f')$, then there is a unique morphism $f:X\to
Y$ over $S$ such that $f'=p^*(f)$.   
\end{prop}
\begin{proof}
  See \cite[A.III.1 Lemma]{freitag-kiehl}.
\end{proof}\ 
 
We now describe the flat descent for objects. For any two integers
$1\le i<j\le 3$, 
denote by $p_{ij}:S^{(3)}\to S^{(2)}$ 
the projection map onto the $i$-th and
$j$-th components. 
Let $p^{n}_i:S^{(n)}\to S'$
denote the $i$-th projection map. Clearly, one has 
$p_1 p_{ij}=p^{3}_i$ and $p_2 p_{ij}=p^{3}_j$ in
$\Hom(S^{(3)},S')$. If $X'/S'$ is a scheme over $S'$ and
$\alpha:p_1^*(X')\to p_2^*(X')$ is a morphism over $S^{(2)}$, 
then the pull-back morphism $p_{ij}^*(\alpha)$
is a morphism 
\[ p_{ij}^*(\alpha):(p_i^{3})^ *(X')\to (p_j^{3})^*(X'). \]

\begin{defn}\label{chow.3}\

  {\rm (1)}
  Let $p:S'\to S$ be a faithfully flat and quasi-compact morphism of
  schemes. 
  A {\it descent datum} for $p$ consists of a pair $(X'/S',\alpha)$, 
  where
  $X'/S'$ is a scheme over $S'$ and $\alpha:p_1^*(X')\isoto p_2^*(X')$
  is an isomorphism of schemes over $S''$ satisfying the condition
  \begin{equation}
    \label{eq:chow.1}
    p_{23}^*(\alpha)\circ p_{12}^*(\alpha)=p_{13}^*(\alpha). 
  \end{equation}

  {\rm (2)} A descent datum $(X'/S',\alpha)$ is said to be {\it
  effective} if there exit a scheme $X/S$ over $S$ and an isomorphism
  $p^*(X)\simeq X'$ over $S'$. 
\end{defn}

\begin{thm}\label{chow.4}
  Let $(X'/S',\alpha)$ be a descent datum for a faithfully flat and
  quasi-compact morphism $p:S'\to S$ of base schemes. If $X'/S'$ is
  quasi-affine, then $(X'/S',\alpha)$ is effective. 
\end{thm}
\begin{proof}
  See \cite[A.III.6 Proposition]{freitag-kiehl}.
\end{proof}
 
\begin{remark}\label{chow.5}
  A classical Weil descent states that if $p:S'\to S$ is 
  $\Spec K\to \Spec k$
  for an algebraic separable field extension $K/k$ and $X'$ is
  a quasi-projective algebraic variety over $K$, 
  then any descent datum $(X'/K',\alpha)$ is effective. 
  Comparing Weil's descent and Theorem~\ref{chow.4}, one may
  ask whether the assumption of $X'$ in Grothendieck's flat descent 
  can be weakened by assuming only that $X'$ is
  quasi-projective. However, this is not the case. Indeed, 
  there exists an \'etale covering $S'\to S$ of
  schemes and a descent datum $(X'/S',\alpha)$ relative to 
  $S'\to S$ such that $X'\to S'$ is projective, but the descent 
  datum is not effective in the category of schemes. See 
  \cite[Tag 08KF]{stacks} for a counterexample.       
\end{remark}

\subsection{Proof of Theorem~\ref{1.2}}
\label{sec:chow.2}
Recall that a semi-abelian variety is a connected commutative smooth 
algebraic group $G$ which is an extension of abelian variety by an
algebraic torus, that is, the affine subgroup of $G$ is an algebraic
torus.  

Recall the statement of Theorem~\ref{1.2} that $X$ and $Y$ are 
two semi-abelian varieties over a field $k$, and
$K/k$ is a primary field extension. We must show that any morphism
$f:X_K\to Y_K$ over $K$ is defined over $k$. 

Let $p:S:=\Spec K\to \Spec k$ and $p_1,p_2:S'':=S\times_{Spec k} 
S\to S$ be the projection maps. Put $K':=K\otimes_k K$. 
Since $K/k$ is a primary extension, the
scheme $S''$ is irreducible and hence connected. Now let $f\in
\Hom_{K}(X_K,Y_K)$. By Proposition~\ref{chow.2}, it suffices to show
that $p_1^*(f)=p_2^*(f)$. 

Let $x=\Delta:\Spec K=S\to S''=S\times_{\Spec k} S$ be the $K$-valued point of
$S''$ defined by the diagonal morphism. As $p_i\circ \Delta=id$,
one has $x^* p_1^*(f)= x^* p_2^*(f)$, i.e. the morphisms $p_1^*(f)$
and $p_2^*(f)$ agree on the fiber over the point $x$. 
Let $\ell$ be any prime
different from $\char k$. The morphism $p_i^*(f):X_{K'}\to Y_{K'}$
induces a morphism $X_{K'}[\ell^n]\to Y_{K'}[\ell^n]$, where
$X_{K'}[\ell^n]$ denotes the $\ell^n$-torsion finite subgroup scheme
of $X_{K'}$. 
Since $X_{K'}[\ell^n]$ has order prime to $\char k$, it is a finite
\'etale group scheme.  
Denote by
$p_i^*(f)[\ell^n]$ the restriction of $p_i^*(f)$ to the finite group
scheme $X_{K'}[\ell^n]$. As $p_1^*(f)_x=p_2^*(f)_x$, one has 
 $p_1^*(f)[\ell^n]_x=p_2^*(f)[\ell^n]_x$. Since $X_{K'}[\ell^n]$ is
 finite \'etale over ${K'}$ and $S''$ is connected, 
the rigidity of \'etale morphisms \cite[I. Corollary 3.13, p.~26]{milne:ec} 
shows that $p_1^*(f)[\ell^n]=p_2^*(f)[\ell^n]$. Now the collection
$\{X_{K'}[\ell^n]\}_n$ forms a Zariski dense subset of $X_{K'}$, and it
follows that $p_1^*(f)=p_2^*(f)$. This proves Theorem~\ref{1.2}.\qed

\subsection{A descent lemma}
\label{sec:chow.3}
The purpose of this subsection is to prove another descent result. 
This yields a second and simpler proof of
Theorem~\ref{1.2} and hence yields another proof of Theorem~\ref{1.1}. 

\begin{lemma}\label{chow.6}
  Let $X$ and $Y$ be $k$-schemes locally of finite type. Let
  $\{X_n\}_{n\ge 1}$
  be a sequence of closed $k$-subschemes of $X$. Suppose that the
  scheme-theoretic closure of the image $\coprod
  X_n \to X$ is equal to $X$. Let $K/k$ be a field extension, and
  $f:X\otimes_k K \to Y\otimes_k K$ a $K$-morphism. If the morphisms
  $f_n:=f|_{X_n}: 
  X_n\otimes K\to Y\otimes K$ are defined over $k$ for all $n$, then
  $f$ is defined over $k$. 
\end{lemma}
\begin{proof}
  We first show that we can reduce the statement 
  to the case where both $X$ and $Y$ are
  affine. Let $U_i$ and $V_i$ be affine coverings of $X$ and $Y$,
  respectively, such
  that $f(U_i\otimes K)\subset V_i\otimes K$. 
  Clearly, $\{X_n\cap U_i\}_{n\ge 1}$ is a sequence of closed
  subschemes of $U_i$ satisfying the same condition of the lemma. 
  If each morphism 
  $f_i:=f|_{U_i\otimes K}$ is defined over $k$, then we can glue $f_i$
  to be a map $g$ which is defined over $k$ and one has $g\otimes K=f$. 

  Write $X=\Spec A$ and $Y=\Spec B$. Let $I_n$ be the ideal of $A$
  defining the closed subscheme $X_n$. The map $f$ is given by a map
  also denoted by
  $f: B\otimes K \to A\otimes K$. The assumptions say that 
  the induced map
  $f_n: B\otimes_k K \to (A/I_n)\otimes_k K$ is defined over $k$, that is,
  $f_n(B)\subset A/I_n$, and that the natural map $A\to \prod_n A/I_n$ is
  injective. Since the image $f(B)$ is contained in $A\otimes K$ and 
  $\prod_n A/I_n$, it is contained in $A$. This proves the lemma. \qed 
\end{proof}

\subsection{Second proof of Theorem~\ref{1.2}}
\label{sec:chow.4}

Let $f\in \Hom_K(X_K,Y_K)$. Let $\ell$ be a prime different from
$\char k$. 
Since $X[\ell^n]$ and $Y[\ell^n]$ are finite \'etale
group schemes, the functor $\calH(S):=\Hom_S(X[\ell^n]\times
S,Y[\ell^n]\times S)$ for any $k$-scheme $S$ is representable by a
finite $\Z/\ell^n$-module scheme over $k$. In particular,
one has $\calH(K)=\calH(k)$ for any primary field extension $K/k$. 
Thus, the restriction of $f$ to $X[\ell^n]\otimes_k K$ is defined over
$k$ for any $n$. Since the collection $\{X[\ell^n]\}_{n\ge 1}$ 
of finite \'etale group subschemes forms a Zariski dense
subset of $X$ and $X$ is reduced, 
it follows from Lemma~\ref{chow.6} that $f$ is defined over $k$. \qed

\section{Bounds for splitting fields of tori}
\label{sec:spl}

\subsection{Splitting fields}
\label{sec:spl.1}
Let $T$ be an algebraic torus  of dimension $d$ over a field $k$. 
The group of cocharacters of $T$, denoted $X_*(T)$, 
is a free $\Z$-module of rank $d$
equipped with a continuous action of the Galois group
$\Gamma_k:=\Gal(k_s/k)$. Thus, one has a group homomorphism
\begin{equation}\label{eq:sp1.1}
  \rho_T: \Gamma_k \to \Aut(X_*(T))\simeq \GL_d(\Z)
\end{equation}
for a choice of a basis of $X_*(T)$.
The spitting field of $T$ by definition is the smallest field
extension $k_T$ of $k$ such that $T$ splits over $k_T$. The field $k_T$ is
characterized by the property $\ker
\rho_T=\Gal(k_s/k_T)=:\Gamma_{k_T}$. Thus, the map 
$\rho_T$ induces a faithful
representation of 
$\Gal(k_T/k)$ on $X_*(T)$. In particular, $k_T$ is a finite Galois
extension of $k$. For studying algebraic tori, it is useful
to bound the degree of the splitting field of an
algebraic torus. 

\def\Max{{\rm Max}}

For any positive integer $d\ge 1$, let $\Max(d,\Q)$ denote the maximal
order of finite subgroups in $\GL_d(\Q)$. 
Clearly, one has
$[k_T:k]\le \Max(d,\Q)$ for any $d$-dimensional algebraic torus $T/k$.
The following lemma provides explicit bounds for $[k_T:k]$. 

For any integer $N\ge 1$, let $T[N]$ denote the $N$-torsion finite
group subscheme of $T$. When $N$ is prime-to-$\char k$, let $k(T[N])$
be the field extension of $k$ in $k_s$ joining all the coordinates
of points in $T[N](k_s)$, and let
\begin{equation}\label{eq:sp1.2}
  \rho_{T,N}: \Gamma_k\to \Aut(T[N]_{k_s})\simeq \GL_d(\Z/N\Z). 
\end{equation}
Clearly, $k(T[N])$ is the Galois separable extension with
$\Gamma_{k(T[N])}=\ker \rho_{T,N}$.

\begin{lemma}\label{spl.1}
  Let $T$ be a $d$-dimensional algebraic torus over $k$.

{\rm (1)} For any prime-to-$\char k$ positive integer $N$ with 
  $N\ge 3$, one has $k_T\subset k(T[N])$ and $[k_T:k]\,|\, \#
  \GL_d(\Z/N\Z)$. 

{\rm (2)} If $\char k\neq 2$, then $[k_T:k]\,|\, (2\cdot \# \GL_d(\F_2))$. 
\end{lemma}
\begin{proof}
  This follows from the fact 
  that the reduction map $\GL_d(\Z)\to\GL_d(\Z/N\Z)$
  induces an injective map on any finite subgroup $G$ if $N\ge 3$, and
  a map $\rho:G\to \GL_d(\Z/N\Z)$ with $\ker \rho \cap G\subset \{\pm
  1\}$ if $N=2$. This fact follows 
  immediately from a lemma of Serre \cite[Lemma, p.~207]{mumford:av}. \qed  
\end{proof}

Lemma~\ref{spl.1} (1) states that if $N\ge 3$ with $(\char k, N)=1$
and all $N$-torsion points of $T$ are defined over $k$, then $T$
splits over $k$. This reminds a theorem of Raynaud, stating that 
if an abelian variety and all its $N$-torsion points are defined over
$k$, and $N\ge 3$, then the abelian variety has semistable reduction
away from $N$ (cf.~\cite[p.~403]{silverberg-zarhin:1995}).

\begin{defn}\label{spl.2}
  We say that a field $k$ is {\it Hilbertian} if it satisfies one of the
  following variants of the Hilbert irreducibility property:
  \begin{itemize}
  \item [(a)] For any irreducible and separable polynomial
     $f(x,t)=a_d(t)x^d+\dots+ a_0(t)\in k(t)[x]$ 
     over $k(t)$ of degree $d\ge
    1$, where $x$ and $t$ are indeterminates, 
    there exist infinitely many specializations $t=t_0\in k$ such
    that $f(x,t_0)$ is an irreducible and separable polynomial over
    $k$ of degree $d$.  
  \item [(b)] For any $n\ge 1$ and any finite separable
   extension $K_t/k(t_1,\dots, t_n)$ of 
  a rational function field $k(t_1,\dots, t_n)$ of
transcendental degree $n$, there exist infinitely many
specializations $t\leadsto t_0\in k$ such that $K_{t_0}$ is a 
finite separable
extension of $k$ of the same degree $[K_t:k(t_1,\dots,t_n)]$. 
  \end{itemize}
\end{defn}

The conditions (a) and (b) are equivalent \cite[\S 9.5, Remark
(1)]{serre:mw}. 
It is well known that any global field is Hilbertian 
(see \cite[Theorem 13.4.2]{fried-jarden:3}). 
If $k$ is any
field, then $k(t)$ and any finitely generated extension of it
are Hilbertian (see \cite[Theorem 13.4.2]{fried-jarden:3},
also see p.~155 and Theorem 12.10 in the first edition.)

We shall prove
\begin{thm}\label{spl.3}
  For any $d\ge 1$ and any Hilbertian
  field $k$ of \ch zero,
  there exist infinitely many
  Galois extensions $K/k$ with group isomorphic to a finite subgroup
  $G\subset \GL_d(\Q)$ of order $\Max(d,\Q)$.
\end{thm}

As an immediate consequence of Theorem~\ref{spl.3}, we
attain the best bound for $[k_T:k]$.

\begin{cor}\label{spl.4}
  For any $d\ge 1$ and any number field $k$, 
  there exist infinitely many $d$-dimensional algebraic tori $T$
  over $k$ such that $[k_T:k]=\Max(d,\Q)$. 
\end{cor}


\subsection{Proof of Theorem~\ref{spl.3}} 
\label{sec:spl.2}

According to \cite{feit:GLnQ}, the signed permutation group
$\{\pm 1\}^d \rtimes S_d\subset \GL_d(\Q)$ attains 
the maximal order of finite subgroups
of $\GL_d(\Q)$ except for $d\in \{2,4,6,7,8,9,10\}$.  
The exceptional cases are listed in Table 1 (see 
\cite[Table 1]{BDEPS}) with maximal-order finite
subgroups. Here $W(D)$ denotes the Weyl group of the root system with
Dynkin diagram $D$.

\begin{table}[ht]
\begin{center}
\bigskip\begin{tabular} { r  | l | r}

$ d$   & Maximal-order subgroup $G$
  & $\Max(d,\Q)=\# G$\\
     \hline
     2      & $W(G_2)$                   & 12\\
     4      & $W(F_4)$                   & 1152 \\
     6      & $\langle W(E_6),-I\rangle$ & 103680\\
     7      & $W(E_7)$                   & 2903040\\
     8      & $W(E_8)$                   & 696729600 \\
     9      & $W(E_8)\times W(A_1)$      & 1393459200\\
    10      & $W(E_8)\times W(G_2)$      & 8360755200\\ 
all other $d$  & $W(B_d) = W(C_d) = \{\pm 1\}^d \rtimes S_d$ & $2^d
  d!$ \\ 
    \end{tabular}
    \vspace*{1ex}
    \caption{Maximal-order finite subgroups of $\GL_d(\Q)$} \label{Ta1}
\end{center}\end{table}

Theorem~\ref{spl.3} follows from the following proposition.

\begin{prop}\label{spl.5}
  Let $G$ be the finite subgroup as in Table 1. For any Hilbertian
  field $k$ of \ch zero, there exist infinitely many finite Galois
  extensions $K/k$ 
  with group isomorphic to $G$.  
\end{prop}
\begin{proof}
  Note that $G$ is a finite reflection group $W$ except $d=6$.
  Regarding $\GL_n(\Q)=\GL(V)$ and putting $V_k=V\otimes_\Q k$, 
  where $V=\Q^n$ and $V_k=k^n$, one has
  $\GL_n(k)=\GL(V_k)$. In this case, 
  $W \subset \GL_n(V_k)$ is also a finite reflection group acting on
  $V_k$. By \cite[Proposition 6]{BDEPS} (mainly following from 
  Chevalley's theorem \cite[Theorem (A)]{chevalley:weyl}), 
  the fixed subfield 
  $k(x_1,\dots,x_d)^G=k(I_1,\dots, I_d)$ is a purely transcendental
  extension of $k$. By the Hilbert irreducibility property, there
  exist infinitely many finite Galois
  extensions $K$ of $k$ with Galois group isomorphic to $G$. \qed  

  
\end{proof}

\subsection{Remarks on classification of tori}
\label{sec:spl.3}
\def\T{{\rm Tp}}
Consider pairs $(\Gamma,\rho)$ which consist of a finite group $\Gamma$
together with a group monomorphism $\rho:\Gamma\to\GL_d(\Z)$ for some
positive integer $d$. We call $d$ the degree of $(G,\rho)$ 
(or of $\rho$).   
Two such pairs $(\Gamma_i,\rho_i:\Gamma_i\to \GL_{d_i}(\Z))$ 
($i=1,2$) are said to have {\it the same 
  type} 
if $d_1=d_2$ and there exist an isomorphism
$\alpha:\Gamma_1\isoto \Gamma_2$ and an element $g\in \GL_{d_1}(\Z)$
such that $g \rho_1(\gamma)g^{-1}=\rho_2( \alpha(\gamma))$ for all
$\gamma\in \Gamma_1$. 
We say that two integral representations $(M_1, \rho_1)$ and
$(M_2,\rho_2)$ of $\Gamma$ of finite rank have {\it the same type} if
there exists an automorphism $\alpha$ of $\Gamma$ such that
$(M_1,\rho_1\circ \alpha)\simeq (M_2,\rho_2)$ as $\Z\Gamma$-lattices.     
Let $(\Z^d,\rho)$ be the integral representation
of $\Gamma$ associated to $(\Gamma,\rho)$. Then two integral
representations $(\Z^{d_1},\rho_1)$ and $(\Z^{d_2},\rho_2)$ of
$\Gamma$ associated
to $(\Gamma,\rho_1)$ and $(\Gamma,\rho_2)$ have the same type 
if and only if so do $(\Gamma,\rho_1)$ and $(\Gamma,\rho_2)$.
In general, there may exist non-isomorphic faithful integral
representations of a finite group $\Gamma$ which have the same type. 
Let ${\rm Tp}$ (resp.~${\rm Tp}_d$) be the set of types of such pairs
$(\Gamma,\rho)$ (resp.~those of degree $d$). It is easy to see that the map
$(\Gamma,\rho)\mapsto \rho(\Gamma)$ induces a bijection between ${\rm
  Tp}_d$ and the set of $\GL_d(\Z)$-conjugacy classes of finite
subgroups of $\GL_d(\Z)$.    

Similarly, we consider pairs
$(\Gamma,\rho_\Q:\Gamma\embed \GL_d(\Q))$ and define two pairs to have
the same type in the same way.  
Let ${\rm Tp}_{\Q}$ (resp.~${\rm Tp}_{d,\Q}$) be the set of types 
of all such pairs $(\Gamma,\rho_\Q)$ (resp.~those of degree $d$). 
Similarly, we can identify the set ${\rm Tp}_{d,\Q}$ with that of
$\GL_d(\Q)$-conjugacy classes of finite subgroups of $\GL_d(\Q)$.
For each pair $(\Gamma,
\rho:\Gamma\embed \GL_d(\Z))$ we
denote by $\rho_\Q:\Gamma\to \GL_d(\Q)$ the map induced from the
inclusion $\GL_d(\Z)\subset \GL_d(\Q)$. This gives rise to a natural map
${\rm Tp}\to {\rm Tp}_\Q$. This map is surjective, because 
any finite subgroup of $\GL_d(\Q)$ preserves a lattice of $\Q^d$ and
then it is conjugate to a subgroup of $\GL_d(\Z)$.

To any algebraic torus $T$ over a field $k$ we associate a triple
$(k_T,\Gal(k_T/k), \rho_T)$, where 
\begin{itemize}
\item $k_T$ is the splitting field of $T$,
\item $\Gal(k_T/k)$ is the Galois group of $k_T/k$, and
\item $\rho_T:\Gal(k_T/k)\to \GL_d(\Z)$ ($d=\dim T$) 
is the faithful representation induced by \eqref{eq:sp1.1}.
\end{itemize}
The map $\rho_T$ is only well-defined up to $\GL_d(\Z)$-conjugacy. 
It is  known that 
the triple $(k_T,\Gal(k_T/k), \rho_T)$ determines
$T$ up to $k$-isomorphism.
Clearly the type $[(\Gal(k_T/k), \rho_T)]$ of 
$(\Gal(k_T/k), \rho_T)$ is well-defined, which we call
the \emph{type} of $T$. 
Thus, the association $[(\Gal(k_T/k),\rho_T)]$ to $T$ induces a
well-defined map ${\rm Tori}_k\to {\rm Tp}$, where 
${\rm Tori}_k$ is the set of 
isomorphism classes of $k$-tori. 
The pre-image of each type $[(\Gamma,\rho)]$ consists of all pairs
$(k',\rho')$, where $k'$
runs through finite Galois extensions $k'$ of $k$ in $k_s$ such that
$\Gal(k'/k)\simeq \Gamma$, and $\rho'$ runs through non-isomorphic
integral representations of $\Gamma$ of the same type as $\rho$. 
Let ${\rm Tori}_k^{\rm isog}$ be the set of 
isogeny classes of $k$-tori.
Similarly the 
association $[(\Gal(k_T/k),\rho_{T,\Q})]$ to $T$ 
induces a well-defined map
${\rm Tori}_k^{\rm isog}\to {\rm Tp}_\Q$. 
The pre-image of
$[(\Gamma,\rho_\Q)]$ consists of pairs $(k',\rho'_\Q)$ with $k'$ as
above and $\rho_\Q'$ running over non-isomorphic $\Q$-representations of
$\Gamma$ that have the same type as $\rho_\Q$.
Thus, classifying ${\rm Tori}_k$
and ${\rm Tori}_k^{\rm isog}$ can be reduced essentially to 
classifying $\T$ and $\T_\Q$, respectively,
and solving the inverse Galois problem.   

The sets $\T_{d}$ and $\T_{d,\Q}$ have been classified for lower degrees
$d$. 
It is obvious that $|\T_1|=2$ and $|\T_{1,\Q}|=2$. 
Two-dimensional tori have been classified by
Voskresenskii~\cite{Vos:tori} (cf.~Seligman~\cite{seligman});
particularly they prove $|\T_2|=13$.  
See \cite[Lemma 4.7]{kunyavaskii-sansuc:2001} for a list of 9
indecomposable subgroups in $\GL_2(\Z)$ (labeled $G_1,\dots, G_9$
there). The remaining 4 decomposable subgroups are 
$C_1=\<\diag(1,1)\>$, $C_2=\<\diag(-1,-1)\>$, $C_2'=\<\diag(1,-1)\>$
and $C_2\times C_2=\{\diag(\pm 1,\pm 1)\}$. One can check using 
characters that $C_2'$ and $G_1$, $C_2\times C_2$ and $G_2$, $G_5$ and
$G_6$ are conjugate in $\GL(2,\Q)$. Thus, $|\T_{2,\Q}|=10$. 
Tahara \cite{tahara} classifies finite subgroups of
$\GL(3,\Z)$ up to conjugacy. However, 
there are overlapping 2 same classes in Tahara's list (corrected in 
Ascher and Grimmer \cite{ascher-grimmer}) and
one has $|\T_3|=73$. According to \cite{crystal}, we know $|\T_4|=710$.
Conjugacy classes of finite subgroups of $\GL_d(\Q)$ 
for $d\le 4$ are classified in \cite{BBNWZ}. 
From this we have 
$|\T_{3,\Q}|=32$ and $|\T_{4,\Q}|=227$; 
also see \cite[pp.~54, 69]{kang-zhou}. We make a short list:
\begin{center}
\begin{tabular}{|c|c|c|c|c|c|c|}  \hline
$d$         & $1$ & $2$ & $3$   & $4$   & $5$    & $6$    \\ \hline
$\T_{d,\Q}$ & $2$ & $10$ & $32$ & $227$ & $955$  & $7103$ \\ \hline
$\T_d$      & $2$ & $13$ & $73$ & $710$ & $6079$ & $85308$\\ \hline
\end{tabular}   
\end{center}
(our reference for $d=5,6$ is \cite[Section 3]{hoshi-yamasaki}).
These explicit classifications have been
used to make further development of the Noether problem for finite
subgroups of $\GL_d(\Q)$ (for $d=3,4$); see \cite{kang-zhou} 
and the references therein for details. These also play an important
role in the fundamental work of Hoshi and Yamasaki
\cite{hoshi-yamasaki} on 
the rationality problem of tori of dimension up to $5$. 
  
\section*{Acknowledgments}
This paper grew out from lectures the author gave in the 2017 Fall 
NCTS course. He thanks the audience, Nai-Heng Sheu and Jiangwei Xue 
for their input and discussions. He is grateful to Ming-Chang
Kang and Boris Kunyavskii 
for pointing out mistakes and helpful comments on an earlier
venison of this paper. 
The author is partially supported by the MoST grants 
104-2115-M-001-001-MY3 and 107-2115-M-001-001-MY2. 
He thanks the referee for a careful
reading and helpful comments.






\def\jams{{\it J. Amer. Math. Soc.}} 
\def\invent{{\it Invent. Math.}} 
\def\ann{{\it Ann. Math.}} 
\def\ihes{{\it Inst. Hautes \'Etudes Sci. Publ. Math.}} 

\def\ecole{{\it Ann. Sci. \'Ecole Norm. Sup.}}
\def\ecole4{{\it Ann. Sci. \'Ecole Norm. Sup. (4)}} 
\def\mathann{{\it Math. Ann.}} 
\def\duke{{\it Duke Math. J.}} 
\def\jag{{\it J. Algebraic Geom.}} 
\def\advmath{{\it Adv. Math.}}
\def\compos{{\it Compositio Math.}} 
\def\ajm{{\it Amer. J. Math.}} 
\def\crelle{{\it J. Reine Angew. Math.}}
\def\plms{{\it Proc. London Math. Soc.}}
\def\jussieu{{\it J. Inst. Math. Jussieu}} 
\def\grenoble{{\it Ann. Inst. Fourier (Grenoble)}}
\def\imrn{{\it Int. Math. Res. Not.}}
\def\tams{{\it Trans. Amer. Math. Sci.}}
\def\mrl{{\it Math. Res. Lett.}}
\def\cras{{\it C. R. Acad. Sci. Paris S\'er. I Math.}} 
\def\mathz{{\it Math. Z.}} 
\def\cmh{{\it Comment. Math. Helv.}}
\def\docmath{{\it Doc. Math. }}
\def\asian{{\it Asian J. Math.}}
\def\acta{{\it Acta Math.}}
\def\indiana{{\it Indiana Univ. Math. J.}}

\def\acad{{\it Proc. Nat. Acad. Sci. USA}}

\def\jlms{{\it J. London Math. Soc.}}
\def\blms{{\it Bull. London Math. Soc.}}
\def\manmath{{\it Manuscripta Math.}} 
\def\jnt{{\it J. Number Theory}} 
\def\ijm{{\it Israel J. Math.}}
\def\ja{{\it J. Algebra}} 
\def\pams{{\it Proc. Amer. Math. Sci.}}
\def\smfmemoir{{\it Bull. Soc. Math. France, Memoire}}
\def\bsmf{{\it Bull. Soc. Math. France}}
\def\sb{{\it S\'em. Bourbaki Exp.}}
\def\jpaa{{\it J. Pure Appl. Algebra}}
\def\jems{{\it J. Eur. Math. Soc. (JEMS)}}
\def\jtokyo{{\it J. Fac. Sci. Univ. Tokyo}}
\def\cjm{{\it Canad. J. Math.}}
\def\jaums{{\it J. Australian Math. Soc.}}
\def\pspm{{\it Proc. Symp. Pure. Math.}}
\def\ast{{\it Ast\'eriques}}
\def\pamq{{\it Pure Appl. Math. Q.}}
\def\nagoya{{\it Nagoya Math. J.}}
\def\forum{{\it Forum Math. }}
\def\tjm{{\it Taiwanese J. Math.}}
\def\rt{{\it Represent. Theory}}
\def\bordeaux{{\it J. Th\'eor. Nombres Bordeaux}}
\def\ijnt{{\it Int. J. Number Theory}}
\def\jmsj{{\it J. Math. Soc. Japan}}
\def\rims{{\it Publ. Res. Inst. Math. Sci.}}
\def\ca{{\it Comm. Algebra}}
\def\osaka{{\it Osaka J. Math.}}
\def\bams{{\it Bull. Amer. Math. Soc.}}

\def\tp{{To appear in }}

\newcommand{\princeton}[1]{Ann. Math. Studies #1, Princeton
  Univ. Press}

\newcommand{\LNM}[1]{Lecture Notes in Math., vol. #1, Springer-Verlag}

\end{document}